\documentclass[11pt]{article}

\usepackage{amssymb}
\usepackage{amsthm}
\usepackage{amsmath}

\usepackage{color}

\def \tilde{\widetilde}

\newcommand{\st}[1]{\ensuremath{^{\scriptstyle \textrm{#1}}}}

\newtheorem{theorem}{Theorem}
\newtheorem*{theorem*}{Theorem}
\newtheorem{lemma}{Lemma}
\newtheorem{proposition}{Proposition}
\newtheorem*{proposition*}{Proposition}

\newtheorem{corollary}{Corollary}
\theoremstyle{definition}
\newtheorem*{definition*}{Definition}

\newtheorem*{comments*}{Comments}

\newtheorem*{example*}{Example}

\theoremstyle{remark}
\newtheorem{remark}{Remark}
\newtheorem*{remarks*}{Remarks}

\newcommand{\alphaparenlist}{
  \renewcommand{\theenumi}{\alph{enumi}}%
  \renewcommand{\labelenumi}{(\theenumi)}%
}

\newcommand{\romanparenlistii}{
  \renewcommand{\theenumii}{\roman{enumii}}%
  \renewcommand{\labelenumii}{(\theenumii)}%
}

\makeatletter
\def\@maketitle{\newpage
 \null
 \vskip 2em
 \begin{center}%
  {\Large\bf \@title \par}%
  \vskip 1.5em
  {\normalsize
   \lineskip .5em
   \begin{tabular}[t]{c}\@author
   \end{tabular}\par}%
  \vskip 2em

 \end{center}%
 \par
 \vskip 2.5em}
\makeatother

\begin{document}






\title{Irreducible Continuous Representations of the Simple Linearly Compact n-Lie Superalgebra  of type $S$  }

\author{Carina Boyallian and Vanesa Meinardi\thanks{%
     Ciem - FAMAF, Universidad Nacional de C\'ordoba - (5000) C\'ordoba,
Argentina
\newline $<$boyallia@mate.uncor.edu - meinardi@mate.uncor.edu$>$.}}

 \maketitle

\begin{abstract}
In the present paper we classify all irreducible continuous representations of
the  simple linearly compact n-Lie superalgebra of type $S.$ The
{cla\-ssi\-fi\-ca\-tion} is based on a bijective correspondence
between  the continuous  representations of the n-Lie algebras $S^n$ and continuous
representations of the Lie
algebra of Cartan type
   $S,$    on which some two-sided ideal acts trivially.
\end{abstract}

\vfill \pagebreak

\section{Introduction}
 \noindent In 1985 Filippov \cite{F} introduced a generalization of a Lie
algebra, which he called an $n$-Lie algebra. The Lie bracket  is
taken between $n$ elements of the algebra instead of two. This new
bracket is $n$-linear, anti-symmetric and satisfies
a generalization of the Jacobi identity.\\
 In \cite{F} and several subsequent papers,
\cite{F1}, \cite{K}, \cite{K1}, \cite{L} a structure theory of finite
dimensional $n$-Lie algebras over a field $\mathbb{F}$ of
characteristic $0$ was developed. In \cite{L},  W. Ling proved that
for every $n\geq 3$ there is, up to isomorphism only one finite dimensional simple
$n$-Lie algebra, namely $\mathbb{C}^{n+1}$ where  the $n$-ary operation is
given by the generalized vector product, namely, if  $e_1,\cdots,
e_{n+1}$ is the standard basis of $\mathbb{C}^{n+1},$  the  n-ary bracket
is given by
$$[e_1,\cdots, \hat{e_i},\cdots e_{n+1}]=(-1)^{n+i-1}\, e_i,$$ where
$i$ ranges from $1$ to $n+1$ and the hat means that $e_i$ does not
appear in the bracket.

 A. Dzhumadildaev studied in \cite{D1} the
finite dimensional irreducible {re\-pre\-sen\-ta\-tions} of the simple $n$-Lie
algebra $\mathbb{C}^{n+1}$. D.Balibanu and J. van de Leur in
\cite{BL} classified both, finite and infinite-dimensional
irreducible {hi\-ghest} weight representations of this algebra.
 Another examples of  $n$-Lie algebras appeared earlier in
Nambu's generalization of Hamiltonian {dy\-na\-mics} \cite{N}. A more
recent important example of an $n$-Lie algebra structure on
$C^{\infty}(M),$ where $M$ is a finite-dimensional manifold, was given by Dzhumadildaev in \cite{D}, and it is associated to $n-1$ commuting vector fields $D_1, \cdots, D_{n-1}$ on $M.$ More precisely, it is the
space $C^{\infty}(M)$ of $C^{\infty}$-functions on
 $M,$ endowed with a n-ary bracket,
associated to $n-1$ commuting vector fields $D_1, \cdots, D_{n-1}$ on
$M$:
\begin{equation}
  \label{eq:0.3}
  [f_1 ,\ldots ,f_n] = \hbox{det} \left(
    \begin{array}{ccc}
f_1 & \ldots & f_n\\
D_1 (f_1) & \ldots & D_1 (f_n)\\
\hdotsfor{3}\\
D_{n-1} (f_1) & \ldots & D_{n-1}(f_n)
\end{array} \right) \, .
\end{equation}

 A linearly compact algebra is a topological algebra, whose underlying vector space is linearly compact, namely is a topological product of finite-dimensional vector spaces, endowed with discrete topology (and it is {a\-ssu\-med} that the algebra product is continuous in this topology).
  In $2010$, N. Cantirini and V. Kac,  (\cite{CK}),  classified simple linearly compact $n$-Lie superalgebras with $n>2$
over a field $\mathbb{F}$ of characteristic 0.  The list consists in
four examples, one of them being $n+1$-dimensional vector
product $n$-Lie algebra, and the remaining three are
infinite-dimensional $n$-Lie algebras. More precisely,
\begin{theorem} \rm{\label{cantarini-kac teo}}
  \label{th1}
\alphaparenlist
 \cite{CK} \begin{enumerate}
  \item Any simple linearly compact $n$-Lie
    algebra with $n>2$, over an algebraically closed
    field~$\mathbb{F}$ of characteristic 0, is isomorphic to one of the
following four examples:

\romanparenlistii
    \begin{enumerate}
    \item 
the $n+1$-dimensional vector product $n$-Lie algebra $\mathbb{C}^{n+1}$;

\item 
the $n$-Lie algebra, denoted by $S^n$, which is the linearly compact
vector space of formal power series $\mathbb{F} [[x_1, \ldots
,x_n]]$, endowed with the $n$-ary bracket
 \begin{equation*}
  \label{eq:0.2}
  [f_1,\ldots ,f_n] = \det \left(
 \begin{array}{ccc}
     D_1 (f_1) & \ldots & D_1 (f_n)\\
\hdotsfor{3}\\
D_{n} (f_1) & \ldots & D_{n}(f_n)
\end{array} \right) \, .
\end{equation*}where $D_i =
\frac{\partial}{\partial x_i}$;

\item 
the $n$-Lie algebra, denoted by $W^n$, which is the linearly compact
vector space of formal power series $\mathbb{F} [[x_1,\ldots
,x_{n-1}]]$, endowed with the $n$-ary bracket,
\begin{equation*}
  \label{eq:0.3}
  [f_1 ,\ldots ,f_n] = \hbox{det} \left(
    \begin{array}{ccc}
f_1 & \ldots & f_n\\
D_1 (f_1) & \ldots & D_1 (f_n)\\
\hdotsfor{3}\\
D_{n-1} (f_1) & \ldots & D_{n-1}(f_n)
\end{array} \right) \, .
\end{equation*} where $D_i =
\frac{\partial}{\partial x_i}$;

\item 
  the $n$-Lie algebra, denoted by $SW^n$, which is the  direct
  sum of $n-1$ copies of $\mathbb{F} [[x]]$, endowed with the following
  $n$-ary bracket, where $f^{\langle j \rangle}$ is an element of the
  $j$\st{th} copy and $f' = \frac{df}{dx}$:
\begin{align*}
  [f^{\langle j_1 \rangle}_1 , \ldots f^{\langle j_n \rangle}_n ] =0,
\hbox{\,\, unless \,\,} \{ j_1,\ldots ,j_n \}\supset \{ 1 ,\ldots , n-1 \},\\
  [f^{\langle 1 \rangle}_1 ,\ldots , f^{\langle k-1 \rangle}_{k-1},
    f^{\langle k \rangle}_k, f^{\langle k \rangle}_{k+1},
    f^{\langle k+1 \rangle}_{k+2}, \ldots ,
    f^{\langle n-1 \rangle}_n ]\\
    = (-1)^{k+n} (f_1 \ldots f_{k-1} (f'_kf_{k+1} - f'_{k+1}f_k)
       f_{k+2}\ldots f_n)^{\langle k \rangle}\, .
\end{align*}
    \end{enumerate}
\item 
There are no simple linearly compact $n$-Lie  superalgebras over
$\mathbb{F}$, which are not $n$-Lie algebras, if $n>2$.

\end{enumerate}

\end{theorem}
In the present paper, we aim to classify all irreducible continuous {re\-pre\-sen\-ta\-tion} of the simple linearly compact $n$-Lie
  algebra $S^n.$
In the same way that D.Balibanu and J. van de Leur did the
classification of irreducible modules in \cite{BL}, we reduced the
problem to find irreducible continuous {re\-pre\-sen\-ta\-tions} of simple linearly
compact $n$-Lie (super) algebra $S^n$ to find irreducible
continuous representations
of its associated basic Lie algebra on
which some two-sided ideal acts trivially.
\noindent
The paper is organized as follow: In Section 2 we give the basic definitions and results related with $n$-lie algebras and state the relationship between representation of $n$- Lie  algebras and representations of its associated Lie algebra. In Section 3,  we introduce the simple linearly compact $n$- Lie algebra $S^n$,  we identify its associated Lie algebra with the Lie algebra of its inner derivations which is nothing but $S_{n}$, the Lie algebra of Cartan type $S$ and finally we relate representation of the $n$-Lie algebra $S^n$ with representations of $S_{n}$. In Section 4, we present some general results of the representation theory of $S_{n}$, prove some technical lemmas and we describe some generators of the two sided ideal that must act trivially in our representations. Finally in Section 5, we state and prove the main result of the paper.

\section{$n$-Lie algebras and $n$-Lie modules}

 We will give an introduction to $n$-Lie algebras and
$n$-Lie modules. We will also introduce some useful results over the
correspondence between representations of $n$-Lie algebra and
representations of its basic associated Lie algebra.

From now on, $\mathbb{F}$ is a field of characteristic zero. As mentioned before, we are interested in studying irreducible representations
of the linearly compact $n$-Lie superalgebra $S^n.$ N. Cantarini and
V. Kac stated in
\cite{CK} that there are no simple linearly compact $n$-Lie
superalgebras over $\mathbb{F},$ which are not $n$-Lie algebras. Then
we will use the representation theory of $n$-Lie algebras to give the
representation theory of simple linearly compact $n$-Lie
superalgebras.
Given an integer $n\geq 2$, an $n$-\textit{Lie algebra} $V$ is  a vector space over
a field $\mathbb{F}$, endowed with an $n$-ary anti-commutative product
$$
\begin{aligned}
\wedge^n V\,\,\,\,\, &\longrightarrow\,\,\,\, V\\
a_1\wedge\cdots\wedge a_n &\mapsto[a_1,\cdots,a_n],
\end{aligned}
$$
subject to the following Filippov-Jacobi identity:
\begin{equation}
\label{FJ}
\begin{aligned}
& [a_1, \dots , a_{n-1},[b_1, \dots , b_n]]=  [[a_1, \dots , a_{n-1}, b_1], b_2,\dots , b_n] +\\
&\,\,[b_1, [a_1, \dots , a_{n-1},b_2],b_3,\dots , b_n]+\dots +[b_1,\dots , b_{n-1}, [a_1, \dots , a_{n-1}, b_n]].
\end{aligned}
\end{equation}

A \textit{derivation} $D$ of
an $n$-Lie algebra $V$ is an endomorphism of the vector
space $V$  such that:
$$  D ([a_1 ,\cdots, a_n]) =
    [D (a_1) , a_2 , \cdots , a_n]
    + [a_1 , D(a_2) ,\cdots
    , a_n] + \cdots +
     [a_1 , \cdots , D (a_n)] .$$
As in the Lie algebra case ($n=2$), the meaning of the Filippov- Jacobi identity is that  all
endomorphisms $D_{a_1,\ldots ,a_{n-1}}$ of $V$ ($a_1,\ldots
a_{n-1} \in V$), defined by
\begin{displaymath}
D_{a_1,\ldots a_{n-1}}(a) = [a_1,\ldots ,a_{n-1},a]
\end{displaymath}
are derivations of $V$. These derivations are called \textit{inner}.

A subspace $W \subset V$ is called a $n$-\textit{Lie subalgebra} of the $n$-Lie algebra
$V$ if $[W,\cdots,W] \subset W.$ An $n$-Lie subalgebra $I \subset V$ of an $n$-Lie algebra is
called an \textit{ideal} if $[I,V,\cdots,V] \subset I.$ An $n$-Lie
algebra is called \textit{simple} if it has not proper ideal besides
$0.$

Let $V$ be an $n$-Lie algebra, $n\geq 3$.  We will
associate to $V$ a Lie algebra called \textit{the basic Lie
algebra}, following the presentation given in \cite{D1} and \cite{BL}.
Consider $ad:\wedge^{n-1}V\to \textrm{End}(V)$
given by $ad(a_1\wedge\ldots\wedge
a_{n-1})(b):=D_{a_1,\ldots a_{n-1}}(b)=[a_1,\ldots,a_{n-1},b]$. One can easily see that we
could have chosen the codomain of $ad$ to be $\textrm{Der}(V)$
(the set of derivations of $V$) instead of $\textrm{End}(V)$. $ad$ induces a map $\tilde{ad}:\wedge^{n-1}V\to
\textrm{End}(\wedge^\bullet V)$ defined as
$\tilde{ad}(a_1\wedge\ldots\wedge a_{n-1})(b_1\wedge\ldots\wedge
b_m)=\sum_{i=1}^{m}b_1\wedge\ldots\wedge[a_1,\ldots,a_{n-1},b_i]\wedge\ldots\wedge
b_m$. Denote by $\textrm{Inder}(V)$ the set of inner
derivations
of $V$, i.e. endomorphisms of the form $D_{a_1,\ldots a_{n-1}}=ad(a_1\wedge\ldots\wedge a_{n-1})$.\\
 The set of derivations $\hbox{Der}(V)$ of an $n$-Lie algebra V is a Lie algebra under the commutator and $\hbox{Inner}(V)$ is a Lie ideal. Notice the Lie brackect of $\hbox{Inner}(V)$ can be given by
$$[\hbox{ad}(a_1\wedge\cdots \wedge a_{n-1}), \, \hbox{ad}(b_1\wedge \cdots \wedge b_{n-1})]= \hbox{ad}(c_1\wedge\cdots \wedge c_{n-1}),$$
where
$$c_1\wedge\cdots \wedge c_{n-1}=\sum_{i=1}^{n-1}b_1\wedge\ldots\wedge[a_1,\ldots,a_{n-1},b_i]\wedge\ldots\wedge
b_{n-1}=\tilde{ad}(a)(b).$$
By skew symmetric condition $c_1\wedge\cdots \wedge c_{n-1}$ can be defined also by
$$c_1\wedge\cdots \wedge c_{n-1}=-\sum_{i=1}^{n-1}a_1\wedge\ldots\wedge[b_1,\ldots,b_{n-1},a_i]\wedge\ldots\wedge
a_{n-1}=-\tilde{ad}(b)(a).$$
Then $\tilde{ad}$ is skew-symmetric (cf. \cite{BL2}). We give to $\wedge^{n-1}V$ a Lie algebra structure under the Lie bracket defined by

\begin{equation}[a,b]= \tilde{ad}(a)(b).\end{equation}
 Therefore this proposition follows,
\begin{proposition}\label{propsuryectivo ad}
$[\cdotp,\cdotp]$ defines a Lie algebra structure on
$\wedge^{n-1}V $ and \linebreak $ad:\wedge^{n-1}V\to
\textrm{Inder}(V)$ is a surjective Lie algebra homomorphism.
\end{proposition}
\noindent Consider
$$
\hbox{Ker}\hbox{(ad)}=\{a_1\wedge \cdots \wedge a_{n-1} \in
\wedge ^{n-1}V:\, \hbox{ad}(a_1\wedge \cdots \wedge
a_{n-1})(b)=0 \hbox{ for all}\, b \in V\}
$$
and
$$\hbox{Ker}\tilde{\hbox{(ad)}}=\{a_1\wedge \cdots \wedge
a_{n-1} \in \wedge ^{n-1}V:\, \tilde{\hbox{ad}}(a_1\wedge
\cdots \wedge a_{n-1})(b)=0 \hbox{ for all}\, b \in
\wedge^{\bullet}V \}$$
It is straightforward to check that $\hbox{Ker}(\hbox{ad})$ is an abelian ideal of $\wedge^{n-1}V$ and $\hbox{Ker} (\hbox{ad})\subseteq\hbox{Ker}(\tilde{\hbox{ad}})$. Thus
\begin{equation}\wedge ^{n-1}V/ \hbox{Ker}\hbox{(ad)}\simeq
\hbox{Inder}(V),
\end{equation} as Lie algebras. Thus,
\begin{equation}\label{iso1}\wedge^{n-1}V\simeq \hbox{Ker}\hbox{(ad)}\rtimes
\hbox{Inder}(V).\end{equation}

\,

A vector space $M$ is called an $n$-\textit{Lie module} for the $n$-Lie
algebra $V$, if on the direct sum $V\oplus M$ there is
a structure of  $n$-Lie algebra, such that the following
conditions are satisfied:
\begin{itemize}
\item $V$ is a subalgebra;
\item $M$ is an abelian ideal, i.e. when at least two slots of the
$n$-bracket are occupied by elements in $M$, the result is 0.
\end{itemize}

We have the following results that establish some relations between representations of $\wedge ^{n-1}V$ and $n$-Lie modules.

\begin{theorem}\label{Main theorem}\begin{itemize}\item[1)] Let $M$ be an $n$-Lie module  of the $n$-Lie algebra
$V$ and define $\rho: \wedge^{n-1} V \rightarrow
\hbox{End}(M) $
 given by $$\rho(a_1\wedge \cdots \wedge
a_{n-1})(m):=[a_1,\cdots, a_{n-1}, m]$$ for all $m \in M,$ where
this $n$-Lie bracket corresponds to the $n$-Lie structure of $V\oplus M.$ Then $\rho$ is an homomorphism of Lie algebras.

\item[2)]Given $(M, \rho)$ a representation of $\wedge^{n-1}V$
such that the two sided ideal $Q(V)$ of the universal enveloping algebra
of $\wedge ^{n-1}V,$ generated by the elements
\begin{equation}\label{vanesa condition}x_{a_1, \cdots,
a_{2n-2}}=[a_1, \cdots, a_n]\wedge a_{n+1}\wedge \cdots a_{2n-2}-$$
$$-\sum_{i=1}^n(-1)^{i+n}(a_1 \wedge \cdots \wedge\hat{a_i} \wedge \cdots \wedge a_n)(a_i\wedge a_{n+1}\wedge \cdots \wedge a_{2n-2})\end{equation}
acts trivially on $M,$ then $M$ is an $n$-Lie module.
\end{itemize}
\end{theorem}
\begin{proof}
Part $1)$ is direct from the definition of the Lie bracket in
$\wedge^{n-1}V$ and the Filippov-Jacobi  identity of the
$n$-Lie bracket corresponding to the $n$-Lie structure of the
semidirect product of $V$ and $M.$

Let's prove part $2).$ Consider the $n$-ary map $[[\,,\,]]:
\wedge^{n-1}(V\ltimes M)\rightarrow V\ltimes M$ such
that $M$ is an abelian ideal   and $V$ is a subalgebra with its
own $n$-Lie bracket and define
\begin{equation}\label{doble bracket}[[a_1, \cdots, a_{n-1},
m]]:=\rho(a_1\wedge\cdots \wedge a_{n-1})(m)\end{equation} where
$a_i \in V, \, m \in M.$ We need to show the Filippov-Jacobi
identity holds for the $n$-ary bracket defined above. It is
enough to show that
$$[[a_1,\cdots, a_{n-1},[[b_1,\cdots,
b_{n-1},m]]]]-[[b_1,\cdots, b_{n-1},[[a_1,\cdots, a_{n-1},m]]]]=$$
\begin{equation}\label{first Jacobi identity}
\sum\limits_{i=1}^{n-1}[[b_1,\cdots [a_1,\cdots, a_{n-1}, b_i],
\cdots, b_{n-1}, m]]
\end{equation}
and

 $[[[a_1,\cdots, a_n],a_{n+1},\cdots, a_{2n-2},m]]=$
\begin{equation}\label{second Jacobi identity}\sum\limits_{i=1}^{n-1}(-1)^{n+i+1}[[a_{1},\cdots
[a_{n+1},\cdots, a_{2n-2}, a_{i},m], \cdots, a_{2n-2}]]
\end{equation}
hold for $a_i$ and $b_i \in V$  and $m
\in M.$

Since $\rho$ is a representation of $\wedge^{n-1} V$ and $\rho[a,b]=\rho(\tilde{ad}(a)(b))$ by definition of the Lie bracket, then  the identity (\ref{first Jacobi identity}) holds.

Let's prove the identity (\ref{second Jacobi identity}).
 Writing the identity (\ref{second Jacobi identity}) using (\ref{doble bracket})
we have that
$$\rho([a_1,\ldots,a_n]\wedge a_{n+1}\wedge\ldots\wedge
a_{2n-2})(m)$$
\begin{equation}\label{eq3}=\sum_{i=1}^n(-1)^{i+n}\rho(a_1\wedge\ldots\wedge\hat{a_i}\wedge\ldots\wedge
a_n)\rho(a_i\wedge a_{n+1}\wedge\ldots\wedge a_{2n-2})(m).
\end{equation}
Therefore (\ref{eq3}) is equivalent to the fact that the ideal $Q(V)$ acts
trivially on $M,$ finishing our proof.
\end{proof}
The following Proposition was proven in  \cite{D1}.
\begin{proposition}\label{prop2} Let $M$ be a $n$-Lie module over an $n$-Lie algebra $V.$ Then any submodule,
any factor-module and dual module of $M$ are also $n$-Lie modules.
If $M_1$ and $M_2$ are $n$-Lie modules over $V,$ then their direct
sum $M_1\oplus M_2$ is also $n$-Lie module.
%

\end{proposition}
As in \cite{D1} we deduce the following Corollary.
\begin{corollary}\label{corol1} Let $M$ be a $n$-Lie module over $n$-Lie algebra $V.$ Then
\begin{itemize}
\item [a)] $M$ is irreducible if and only if $M$ is irreducible as a Lie module over Lie algebra $\wedge^{n-1}V.$
\item [b)] $M$ is completely reducible, if only if $M$ is completely reducible as a Lie module over Lie algebra $\wedge^{n-1}V.$
\end{itemize}
\end{corollary}
Since we are aiming the study of the representation theory of $V$ as an $n$-Lie algebra,  Theorem \ref{Main theorem}  shows that it is closely
related to the representation theory of the Lie algebra
$\wedge^{n-1}V.$ But first, due to (\ref{iso1}),
we need to characterize the ideal $\hbox{Ker}\hbox{(ad)}.$ We have
the following Lemma.
\begin{lemma}\label{lem2} If $a \in \hbox{Ker}(\hbox{ad})$ and  $\rho$ is a
representation of $\wedge^{n-1}V,$ then $\rho(a)$ commutes with $\rho(b)$ for any $b \in
\wedge^{n-1}V.$
\end{lemma}
\begin{proof}Consider $a \in \hbox{Ker}(\hbox{ad})\subseteq\hbox{Ker}(\tilde{\hbox{ad}}).$ By definition of Lie bracket in $\wedge^{n-1}V $ follows
$$\rho(a)\rho(b)-\rho(b)\rho(a)=\rho[a,b]=\rho(\tilde{\hbox{ad}}(a)(b))=0.$$
\end{proof}
Thus, we have the following Proposition.
\begin{proposition}\label{Schur lemma} Let $\rho$ be an irreducible representation of
$\wedge^{n-1}V$ in $M$ with countable dimension. Then $\hbox{Ker}(\text{ad})$ acts by
scalars in $M.$
\end{proposition}
\begin{proof} Immediate from the Lemma above and Schur Lemma.
\end{proof}
\begin{theorem} \label{thm3} Let $(M,\rho)$ be an irreducible representation of
$\wedge^{n-1}V$ such that
the ideal $Q(V)$ acts trivially on $M$. Then
\begin{itemize}
\item[a)] $\rho|_{\hbox{Ker}\hbox{(ad)}}:=\lambda\,Id$ with $Id$
the identity map in $\hbox{End}(M)$ and $\lambda \in
(\hbox{Ker}\hbox{(ad)})^{\ast}$  is an $\hbox{Inder}(V)$-module homomorphism (where $\mathbb{F}$ is thought as a trivial
$\hbox{Inder}(V)$-module),
\item[b)] $\rho|_{\hbox{Inder}(V)}$ is an irreducible representation of $\hbox{Inder}(V)$
 such that the ideal
$Q(V)$ acts trivially on $M.$
\item[c)] $\rho=\rho|_{\hbox{Inder}(V)}\oplus \lambda\, Id.$
\end{itemize}
\end{theorem}
\begin{proof} Let's prove part a). If $l \in \hbox{Inder}(V)$ and $a \in {\hbox{Ker}\hbox{(ad)}},$ since
${\hbox{Ker}\hbox{(ad)}}$ is an abelian  ideal, by Lemma \ref{lem2} we have
$0=\rho([l,a])(m)=\lambda([l,a])\,Id(m)$ . Thus
$\lambda$ is an $\hbox{Inder}(V)$-module homomorphism.

 Let's prove part b). Consider $N\varsubsetneq M$ a non-trivial $\hbox{Inder}(V)$-{sub\-re\-pre\-sen\-ta\-tion} of $M$ and take $0 \neq m \in M$ such that $0
\neq \tilde{N}:=\rho(\hbox{Inder}(V))(m)\subseteq N.$ Note if
$a \in {\hbox{Ker}\hbox{(ad)}},$ due to Lemma \ref{lem2} and
Proposition \ref{Schur lemma},\, $\rho(a)
\tilde{N}=\rho(a)\rho(\hbox{Inder}(\frak
g))(m)=\rho(\hbox{Inder}(V)\rho(a)(m)=\lambda(a)\rho(\hbox{Inder}(V))(m)=\tilde{N}.$
{U\-sing} (\ref{iso1}), we can conclude that $0 \neq \tilde{N}$ is a
subrepresentation of $M$ as a  $\wedge^{n-1}V$-module  but $M$ was irreducible by hypothesis
which is a {con\-tra\-dic\-tion}. Part c) is an immediate consequence of
(\ref{iso1}) and Lemma \ref{lem2}.\end{proof}
\section{The simple linearly compact $n$-Lie algebra $S^n$}

We denote by $S^n$ the simple infinite-dimensional
linearly compact $n$-Lie superalgebra , whose
underlying vector space is the linearly compact vector space of formal power
series $\mathbb{F}[[x_1,\cdots,x_{n}]]$ endowed with the following
$n$-ary bracket:
\begin{equation}\label{n-bracket}[f_1, \cdots, f_n]=\hbox{det}\left(
  \begin{array}{cc}
  D_1(f_1)\quad\cdots \quad D_1(f_n)\\
     \cdots\cdots\cdots\cdots\cdots\cdots \cdots \\
    D_{n}(f_1)\quad \cdots \quad D_{n}(f_n) \\
  \end{array}
\right)\end{equation}
 where $D_i=\frac{\partial}{\partial x_i}.$\\

\begin{remark}\label{rmk1} \begin{itemize}
\item [(a)] Consider the $n$-Lie algebra
$S^n$ endowed with the $n$-bracket (\ref{n-bracket}) and the map
$\hbox{ad}: \wedge^{n-1} S^n\rightarrow \hbox{Inder}(S^n),$ which sends  $f_1\wedge\cdots\wedge f_{n-1}\rightarrow
\hbox{ad}(f_1\wedge\cdots\wedge f_{n-1})$. By Proposition 1, it  is an epimorphism  of Lie
algebras and we will show that  in this case,
$$\hbox{ker}(ad) =\hbox{span} \left\{f_1\wedge\cdots \wedge f_{n-1}:\, f_i \in \mathbb{F}, \hbox{ for some}\, 1\leq i \leq n-1 \right\}.$$
Note that  any $ f_1\wedge\cdots \wedge f_{n-1}$ such that $f_i \in \mathbb{F}$ for some $\, 1\leq i \leq n-1 $, clearly is in $\hbox{ker}(ad)$. On the other hand, if we assume that $ f_1\wedge\cdots \wedge f_{n-1}\in \hbox{ker}(ad)$, we have
\begin{equation}\label{determinant}\hbox{ad}(f_1\wedge\cdots\wedge f_{n-1})(f) =\hbox{det}\left(
  \begin{array}{cc}
  D_1(f_1)\quad\cdots \quad D_1(f)\\
    \cdots\cdots\cdots\cdots\cdots\cdots \cdots \\
    D_{n}(f_1)\quad \cdots \quad D_{n}(f) \\
  \end{array}
\right)=0,
\end{equation}
 for any $f\in \wedge^{n-1}S^n.$ Since $f$ is arbitary, we have that at least two of the first $ n-1$ columns of this matrix should be
linearly dependent, in other words, there exist $i< j \in \{1, \cdots, n-1\}$ such that $\bigtriangledown f_i = c\bigtriangledown f_j$ for some $c \in \mathbb{F}$. Since $\mathbb{F}$ is a field of characteristic zero, we can deduce that $f_i = cf_j + k,$ for some  $k\in \mathbb{F}$. Thus $ f_1\wedge\cdots \wedge f_{n-1} = k\, (f_1\wedge\cdots\wedge 1 \wedge\cdots \wedge f_j\wedge\cdots\wedge f_{n-1})$.
\item[(b)]  Let $(M,\rho)$ be an irreducible representation of
$\wedge^{n-1}S^n$ such that
the ideal $Q(S^n)$ acts trivially on $M$. By Theorem 3, parts (a) and (b), we have that $\rho|_{\hbox{Ker}\hbox{(ad)}}:=\lambda\,Id$, with $Id$
the identity map in $\hbox{End}(M)$ and  $\lambda\in(\hbox{Ker}\hbox{(ad)})^*$, and $\rho|_{\hbox{Inder}(S^n)}$ is an irreducible representation of $\hbox{Inder}(S^n)$
 such that the ideal
$Q(S^n)$ acts trivially on $M.$ We will show that
$\lambda = 0.$ Consider $ x_{f_1, \cdots,
f_{2n-2}}$
an element of $Q(S^n)$ 
   such that $f_i \notin \mathbb{F}$ for all $i=1, \cdots, 2n-2$ and $[f_1, \cdots, f_n]
   \in \mathbb{F},$ then \begin{equation*}x_{f_1, \cdots, f_{2n-2}}=[f_1, \cdots, f_n]\wedge f_{n+1}\wedge \cdots f_{2n-2}$$$$-\sum_{i=1}^n(-1)^{i+n}(f_1 \wedge \cdots \wedge\hat{f_i} \wedge \cdots \wedge f_n)(f_i\wedge f_{n+1}\wedge \cdots \wedge f_{2n-2})
  $$$$ \qquad\qquad\qquad\in \hbox{Ker}(ad)+ U(\hbox{Inder}(S^n)).\end{equation*}
  \end{itemize}

Fix $m \in M $. We have,\begin{equation}0=\rho(x_{f_1, \cdots, f_{2n-2}})( m) =\lambda(1\wedge f_{n+1}\wedge \cdots f_{2n-2})( m)\end{equation}
\begin{equation}\label{above}-\sum_{i=1}^n(-1)^{i+n}\rho(\hbox{ad}(f_1 \wedge \cdots \wedge\hat{f_i} \wedge \cdots \wedge f_n))\rho(\hbox{ad}(f_i\wedge f_{n+1}\wedge \cdots \wedge f_{2n-2}))(m).\end{equation}
Note that (\ref{above}) is in the image of the ideal $Q(S^n)$ by the  $\hbox{ad}$ map acting on $m$. Thus, by Theorem 3 (b),  $\sum_{i=1}^n(-1)^{i+n}\rho(\hbox{ad}(f_1 \wedge \cdots \wedge\hat{f_i} \wedge \cdots \wedge f_n))\rho(\hbox{ad}(f_i\wedge f_{n+1}\wedge \cdots \wedge f_{2n-2}))\cdot m=0$ , from where we deduce that  $\lambda(1\wedge f_{n+1}\wedge \cdots f_{2n-2})\cdot m=0$.

Now suppose $\lambda \neq 0$ with $\lambda : \hbox{ker}(ad)\rightarrow \mathbb{F}$ and let $\beta =\{1, \alpha_1, \alpha_2, \cdots\}$ a basis of $\mathbb{F}[[x_1,\cdots, x_n]].$ Then $\alpha = \{1\wedge \alpha_{i_1}\cdots\wedge \alpha_{i_{n-2}}:\, i_1<\cdots<i_{n-2}\}$ is a basis of $\hbox{Ker}(ad).$ Then there exists $\alpha_{i_1}\cdots \alpha_{i_{n-2}}\in \beta$ such that $\lambda(1\wedge \alpha_{i_1}\wedge \cdots\wedge \alpha_{i_{n-2}})\neq 0. $ Choosing $f_{n+1}, \cdots, f_{2n-2}$ as $\alpha_{i_1}\cdots \alpha_{i_{n-2}}$ for $x_{f_1,\cdots, f_{2n-2}}$ we have that $\lambda(1\wedge f_{n+1}\wedge \cdots\wedge f_{2n-2})\cdot m \neq 0,$ which is a contradiction. Then it follows that $\lambda= 0 $.
\end{remark}

\

Denote $W(m,n)$ the Lie superalgebra of continuous derivations of the
tensor product $\mathbb{F}(m,n)$ of the algebra of formal power
series in $m$ even commuting variables $x_1,\dots, x_m$ and the Grassmann
algebra in $n$ anti-commuting odd variables $\xi_1,\dots,\xi_n$.
Elements of $W(m,n)$ can be viewed as linear {di\-ffe\-ren\-tial} operators
of the form
$$X=\sum_{i=1}^m P_i(x,\xi)\frac{\partial}{\partial x_i}+
\sum_{j=1}^n Q_j(x,\xi)\frac{\partial}{\partial \xi_j}, ~P_i,
Q_j\in\mathbb{F}(m,n).$$ The Lie superalgebra $W(m,n)$ is simple
linearly compact (and it is finite-dimensional if and only if
$m=0$).

Now we shall describe S(m, n) a linearly compact subalgebras
of $W(m,n)$.

First, given a subalgebra $L$ of $W(m,n)$, a continuous linear map\linebreak  $Div: L
\rightarrow \mathbb{F}(m,n)$ is called a divergence if the action
$\pi_{\lambda}$ of $L$ on $\mathbb{F}(m,n)$, given by
$$\pi_{\lambda}(X)f=Xf+(-1)^{p(X)p(f)}\lambda f Div X, \,\, X \in L,$$
is a representation of $L$ in $\mathbb{F}(m,n)$ for any $\lambda\in\mathbb{F}$. Note that
$$S'_{Div}(L):=\{X\in L~|~ Div X=0\}$$
is a closed subalgebra of $L$. We denote by $S_{Div}(L)$ its derived subalgebra
 An example
of a divergence on $L=W(m,n)$ is the following, denoted by $div$:
$$div(\sum_{i=1}^m P_i\frac{\partial}{\partial x_i}+
\sum_{j=1}^n Q_j\frac{\partial}{\partial \xi_j})=
\sum_{i=1}^m \frac{\partial P_i}{\partial x_i}+
\sum_{j=1}^n (-1)^{p(Q_j)} \frac{\partial Q_j}{\partial \xi_j}.$$
Hence for any $\lambda\in\mathbb{F}$ we get the representation $\pi_\lambda$ of
$W(m,n)$ in $\mathbb{F}(m,n)$. Also, we get closed subalgebras $S'_{div}(W(m,n))
\supset S_{div}(W(m,n))$ denoted by $S'(m,n)\supset S(m,n)$. Observe that $S^{\prime}(m,n)=S(m,n)$  is simple if $m >1.$  From now on, we will denoted the Lie algebras $S(n, 0)$ by $S_n.$

Proposition 5.1 in \cite{CK} gives  the
description of the Lie algebra of continuous  derivation of each simple
linearly compact $n$-Lie algebra.
Moreover, they state in particular,  that the Lie algebra of continuous derivations of the
$n$-Lie algebra $ S^n$ is isomorphic to
$S_{n}$ and in the proof of this Proposition, they show
that the Lie algebra of continuous derivations of the $n$-Lie
algebra $ S^n$  coincides with the Lie
algebra of its inner derivations. Thus,
\begin{equation}\label{inder}
\hbox{Inder}(S^n)\simeq S_{n}.
\end{equation}
Therefore, Theorems \ref{Main theorem} and \ref{thm3} and Remark \ref{rmk1} gives us the following.

\begin{theorem}\label{th4} Irreducible representations of the $n$-Lie algebra
$S^n$ are in $1-1$ correspondence with irreducible representations of the
universal enveloping algebra $U(S_{n}),$ on which the two sided
ideal $Q(S^n)$, generated by the elements
$$x_{f_1,\cdots,f_{2n-2}}=\hbox{ad }([f_1,\cdots,f_n]\wedge f_{n+1}\wedge\cdots \wedge f_{2n-2} )$$
\begin{equation*}\label{eq2}-\sum_{i=1}^{n} (-1)^{i+n} \hbox{ad }(f_1\wedge \cdots \wedge \widehat{ f_i}\cdots \wedge f_n)\, \hbox{ad }(f_i\wedge f_{n+1}\wedge\cdots\wedge f_{2n-2})
\end{equation*}
 where $f_i \in \mathbb{F}[[x_1,\cdots, x_n]]$ and $f_i \neq 1$  for all $i=1,\cdots, n,$ acts trivially.
\end{theorem}

\section{Representations of simple linearly compact Lie superalgebra $S_{n}.$}

In this section we present the approach given by A.
Rudakov in \cite{R} for the representation theory
of the infinite-dimensional simple linearly compact Lie algebra
$S_{n}$.
The algebra $S_n$ is a subalgebra of the algebra $W_n$ of all derivations of the ring $\mathbb{F}$ of formal power series in $n$ variables. The elements $D \in W_{n}$ has the form $D=\displaystyle{\sum_{i=1}^{n} f_i \partial/\partial x_i}$ with $f_i \in \mathbb{F}[[x_1,\cdots,x_{n}]].$
The algebra $W_n$ is endowed with the filtration
$$(W_{n})_{(j)}=\{D,\, \hbox{ deg }f_i \geq j+1 \}$$
and a compatible gradation
$$(W_{n})_{j}=\{D,\, \hbox{ deg }f_i= j+1 \}.$$
The subalgebra $S_n$ is defined by the condition 
$$\sum_{i=1}^{n}\frac{\partial f_i}{\partial x_i}=0.$$
The filtration and gradation of $W_n$ induce a filtration and gradation in $S_n.$
The gradation of $S_n$ gives a triangular decomposition
\begin{equation*}
S_{n}= (S_{n})_- \oplus (S_{n})_0 \oplus (S_{n})_+,\qquad
\end{equation*}
\hbox{ with } $(S_{n})_\pm=\oplus_{\pm m>0} (S_{n})_m.$
We shall consider continuous representations in spaces with discrete
{to\-po\-lo\-gy}. The continuity of a representation of a {li\-near\-ly}
compact Lie superalgebra $S_{n}$  in a vector space $V$ with
discrete topology means that the stabilizer $(S_{n})_v=\{ g\in
S_{n}  |\, gv=0\}$ of any $v\in V$ is an open (hence of finite
codimension) subalgebra of $S_{n}$. Let $(S_{n})_{\geq
0}=(S_{n})_{
> 0} \oplus (S_{n})_0$. Denote by $P(S_{n}, (S_{n})_{\geq 0})$ the
category  of all continuous $S_{n}$-modules $V$, where $V$ is a
vector space with discrete topology, that are $(S_{n})_0$-locally
finite, that is any $v\in V$ is contained in a finite-dimensional
$(S_{n})_0$-invariant subspace. Given an $(S_{n})_{\geq 0}$-module $F$,
we may consider the {a\-sso\-cia\-ted} induced $S_{n}$-module
\begin{displaymath}
M(F)=\hbox{ Ind}^{S_{n}}_{(S_{n})_{\geq 0}}
F=U(S_{n})\otimes_{U((S_{n})_{\geq 0})} F
\end{displaymath}
called the {\it generalized Verma module} associated to $F$.

 Let $V$ be an $S_{n}$-module. The elements of the subspace
\begin{displaymath}
\hbox{ Sing}(V):=\{ v\in V|\, (S_{n})_{>0} v=0\}
\end{displaymath}
are called {\it singular vectors}. When $V=M(F)$, the $(S_{n})_{\geq0}$-module $F$ is canonically an
$(S_{n})_{\geq0}$-submodule of $M(F)$, and Sing$(F)$ is a subspace
of Sing$(M(F))$, called the {\it subspace of trivial singular
vectors}. Observe that $M(F)= F\oplus F_+$, where
$F_+=U_+((S_{n})_-)\otimes F$ and $U_+((S_{n})_-)$ is the
augmentation ideal in the symmetric algebra $U((S_{n})_-)$. Then
\begin{displaymath}
\hbox{ Sing}_+(M(F)):=\hbox{ Sing}(M(F))\cap F_+
\end{displaymath}
are called the {\it non-trivial singular vectors}.

\begin{theorem}
\label{prop6} \cite{KR1}\cite{R}
 (a) If $F$ is a finite-dimensional
$(S_{n})_{\geq0}$-module, then  $M(F)$ is in $P(S_{n},(S_{n})_{\geq
0})$.

(b) In any irreducible finite-dimensional $(S_{n})_{\geq0}$-module
$F$ the subalgebra $(S_{n})_+$ acts trivially.

(c) If $F$ is an irreducible finite-dimensional
$(S_{n})_{\geq0}$-module, then $M(F)$ has a unique maximal submodule.

(d) Denote by I(F) the quotient  by the unique maximal submodule of
$M(F)$. Then the map $F\mapsto I(F)$ defines a bijective
correspondence between irreducible finite-dimensional
$(S_{n})_{\geq0}$-modules and irreducible $(S_{n})$-modules in
$P((S_{n}),(S_{n})_{\geq0})$, the inverse map being $V\mapsto $Sing$(V)$.

(e) An $(S_{n})_{\geq0}$-module $M(F)$ is irreducible if and only if the
$(S_{n})_{\geq0}$-module $F$ is irreducible and $M(F)$ has no
non-trivial singular vectors.
\end{theorem}
\begin{remark}\label{rmk3} (a) Note that
\begin{equation}(S_{n})_0\cong \frak
sl_{n}(\mathbb{F}),\end{equation}  the isomorphism is given by the
map that sends $x_i \frac{\partial}{\partial x_j} \rightarrow E_{i,j}$ for ($i\neq j$) and $x_i \frac{\partial}{\partial x_i}- x_{i+1} \frac{\partial}{\partial x_{i+1}}\rightarrow E_{i,i}-E_{i+1, i+1},$ where
$E_{i,j}$ denote as  usual the matrix whose $(i,j)$ entry is $1$
and all the other entries are $0$ for  $i, j =1, \cdots, n.$

(b) Due to Theorem \ref{prop6} part b)   any irreducible finite dimensional
$(S_{n})_{\geq0}$-module $F$ will be obtained extending by zero
the irreducible finite dimensional $\frak {sl}_{n}(\mathbb{F})$- module.

\end{remark}
 In the
Lie algebra $\frak{sl}_{n}(\mathbb{F})$ we choose the Borel subalgebra \linebreak$\frak b=
\{ x_i\frac{\partial}{\partial x_i}-x_j\frac{\partial}{ \partial x_j}, \, x_i\frac{\partial}{\partial x_j} : i<j,\, i, j =1,\cdots, n\}.$ We
denote by
$$\frak h=\hbox{span}\{h_i=x_i
\frac{\partial}{ \partial x_i}-x_{i+1}
\frac{\partial}{ \partial x_{i+1}}, \, i=1,\cdots, n-1\} $$ the corresponding Cartan subalgebra.

Let $F_0, \cdots F_{n-1}$ be the irreducible $(S_{n})_{\geq
0}$-modules irreducibles obtained by extending trivially the irreducible
$\frak {sl}_{n}(\mathbb{F})$-modules with highest weight $\lambda_0=(0, 0\cdots,
0),$ $\lambda_1=(1,0, \cdots,  0),$ $\lambda_2=(0,1,\cdots,0),\cdots \lambda_{n-1}=(0,0, \cdots,
 1) $ respectively. We will call them
\textit{exceptional }\,$(S_{n})_{\geq 0}$-modules.

\begin{theorem}\cite{R}\label{th6} Let $F$ be an irreducible finite dimensional $\frak
{sl}_{n}(\mathbb{F})$-module. If $F$ is not isomorphic to one of the exceptional
modules $F_0, \cdots F_{n-1}$ then the $S_{n}$-module $M(F)$ is
irreducible. Each module $N_p:=M(F_p)$ contains a unique irreducible
submodule $K_p$ which is generated by all its non-trivial singular
vectors.
\end{theorem}

\begin{corollary}\cite{R} \label{corol3} If the $S_{n}$-module $E$ is irreducible, then
the $\frak{sl}_{n}(\mathbb{F})$-module $F:=\hbox{Sing}(E)$ is also
irreducible. If $F$ coincides with none of the {mo\-du\-les}
$F_0,\cdots , F_{n-1},$ then $E=M(F).$ If $F=F_p,$ then $E$ is
isomorphic to $J(F_p):=N_p/K_p$.
\end{corollary}

\subsection{Some useful lemmas}
 Let $F$ be an irreducible finite dimensional $\frak{sl}_{n}(\mathbb{F})$-module with highest weight vector $v_{\lambda}$ and highest weight $\lambda.$ Let $J(F)=M(F)/\hbox{ Sing}_+(M(F)).$

Our main goal is to find those  irreducible finite dimensional \linebreak
$\frak{sl}_{n}(\mathbb{F})$-modules $F$ for which $J(F)$ is  an irreducible module
over the $n$-Lie algebra $S^n$, more precisely, we are looking for those $J(F)$ where the ideal $Q(S^n)$ acts trivially.
\begin{lemma}\label{lem3}$Q(S^n)\otimes_{U(S_{n})_{\geq 0}} F \subset \hbox{ Sing}_+(M(F)) $
 if and only if $Q(S^n)$ acts trivially on $\,J(F).$\end{lemma}
 \begin{proof} Suppose $Q(S^n)$ acts trivially on $J(F).$ Note that by Theorem \ref{th6} and Corollary \ref{corol3} this means
 that
 $Q(S^n)\cdot(U(S_{n})\otimes_{U(S_{n})_{\geq 0}} F) \subset \hbox{
 Sing}_+(M(F)).$ In particular,
$ Q(S^n) \otimes_{U(S_{n})_{\geq
 0}}F\subseteq \hbox{Sing}_{+}M(F)$.

Conversely,  if $Q(S^n)\otimes_{U(S_{n})_{\geq 0}} F \subset \hbox{
 Sing}_+(M(F)),$ it is enough to show that  $U(S_{n})Q(S^n)\otimes_{U(S_{n})_{\geq 0}} F \subset \hbox{Sing}_{+}(M(F)).$

 Note that $U(S_{n})(Q(S^n)\otimes_{U(S_{n})_{\geq 0}} F)$ is the submodule generated by $Q(S^n)\otimes_{U(S_{n})_{\geq 0}} F$
 and $Q(S^n)\otimes_{U(S_{n})_{\geq 0}} F \subset \hbox{
 Sing}_+(M(F))$ by hypothesis, thus we have
 $U(S_{n})Q(S^n)\otimes_{U(S_{n})_{\geq 0}} F $ is a submodule  of the unique irreducible submodule $\hbox{Sing}_{+}(M(F))$ of $M(F)$,
 therefore $Q(S^n)$ acts trivially on $J(F).$
 \end{proof}

\begin{lemma}\label{lem4}$Q(S^n)\otimes_{U(S_{n})_{\geq 0}} v_{\lambda} \subset \hbox{ Sing}_+(M(F)) $
 if and only if $\,Q(S^n)$ acts trivially on $J(F).$\end{lemma}
\begin{proof}Due to Lemma \ref{lem3}  we only need to show that  $Q(S^n)\otimes_{U(S_{n})_{\geq 0}} v_{\lambda} \subset \hbox{ Sing}_+(M(F))$  implies that  $\,Q(S^n)$ acts trivially on $J(F).$ It is immediate from the definition of generalized Verma module and the facts that $F$ is a highest weight $\frak
{sl}_{n}(\mathbb{F})$-module and $\frak
{sl}_{n}(\mathbb{F})\subseteq U(S_{n})_{\geq 0}$.
\end{proof}
\subsection{Description of the ideal $Q(S^n)$}
 $\hbox{Inder}(S^n)\simeq
S_{n},$ where the isomorphism is given explicitly by
 \begin{equation}\label{iso2} \hbox{ad}(f_1\wedge\cdots\wedge f_{n-1})\longrightarrow \sum_{i=1}^{n}(-1)^{n+i}\hbox{ det} \left(
\begin{array}{cc}
 D_1(f_1)\, \cdots \,D_1(f_{n-1}) \\
    \cdots\cdots\cdots\cdots\cdots\cdots \cdots \\
    \hat{D_i}(f_1)\, \cdots \,\hat{D_i}(f_{n-1}) \\
    \cdots\cdots\cdots\cdots\cdots\cdots \cdots \\
    D_{n}(f_1)\, \cdots \,D_{n}(f_{n-1}) \\
  \end{array}
\right) D_i,\end{equation} for any $f_1,\cdots, f_{n-1} \in
\mathbb{F}[[x_1,\cdots,x_{n}]], \, D_j=\frac{\partial}{\partial
x_j}$ and the hat means that the $i$-th row  does no appear in the
matrix.
\
\noindent Consider  the subset
$$A=\{D=\sum\limits_{i=1}^{n}f_i
D_i \in S_n: \, f_i \in \mathbb{F}[x_1, \cdots,
x_{n}]\}.$$ It is dense in
$S_{n}$.
Since we are classifying continuous representations  it is enough to characterize a set of generator of $Q_A(S^n):= Q(S^n)\bigcap A$.
Take $f_1, \cdots f_{2n-2} \in \mathbb{F}[x_1, \cdots,x_{n}],$
where $f_l=X^{I_{l}}$ with
$$X^{I_{l}}:=x_1^{i_1^l}x_2^{i_2^l}\cdots
x_{n}^{i_{n}^l},\,$$
where $ I_l:=(i_1^l,\cdots,i_{n}^l)$ with
$i_1^l,\cdots,i_{n}^l \in \mathbb{Z}_{\geq 0}$ and $\, l \in
\{1\cdots,2n-2\}.$ Then the generators of $Q_A(S^n)$ are given by

\begin{equation}\label{eq16}\,x_{f_1,\cdots,f_{2n-2}}=
\left(\sum\limits_{k=1}^{n}\tilde{\alpha}(k)\,
D_k\right)-\sum_{i=1}^{n}(-1)^{i+n}\left(\sum\limits_{q=1}^{n}
\tilde{\beta}(i,q)\,D_q\right)\left(\sum\limits_{s=1}^{n}\tilde{\gamma}(i,s)\,D_s\right),
\end{equation}

\noindent where

 \medskip $\tilde{\alpha}(k)=(-1)^{n+k}\displaystyle{\frac{f_1\cdots
f_{2n-2}}{x_1^2\cdots x_k \cdots x_{n}^2}\; \hbox{ det}\tilde{A}\hbox{
det}\tilde{B}_{k},\quad k=1\cdots,n,}\\$

\

  \medskip$\tilde{\beta}(i,q)=(-1)^{n+q}\displaystyle{\frac{f_1\cdots\hat{f_i}\cdots f_{n}}{x_1\cdots
\hat{x_q} \cdots x_{n}}\hbox{  det}\tilde{A}_{q, i},\quad
q=1\cdots,n,}\\$ \,

\

  $\tilde{\gamma}(i,s)=(-1)^{n+s}\displaystyle{\frac{f_{i}f_{n+1}\cdots
\cdots f_{2n-2}}{x_1\cdots \hat{x_s} \cdots x_{n}}\,
\hbox{ det} \tilde{C}^{(i)}_{s}, \quad s=1\cdots, n,}$

\

\noindent with $i=1,\cdots,n$ and the matrices $\tilde{A}$, $\tilde{B}$ and $\tilde{C}\,$'s are defined as follows:

\
 $$\tilde{A}=\left(
    \begin{array}{cccc}

i_{1}^1&\cdots  & i_{1}^n\\

      \cdots&\cdots&\cdots\\
      \cdots&\cdots&\cdots\\
        i_{n}^1&\cdots & i_{n}^n\\ \\
    \end{array}
  \right),$$
 $\tilde{A}_{q, i}$ is the matrix $\tilde{A}$ with the $q$-row and the $i$-column removed,
%
 \bigskip $$\tilde{B}_{k}=\left(
    \begin{array}{ccccc}
      \sum_{r=1}^{n}(i_1^r-1)&i_1^{n+1}&\cdots&i_1^{2n-2}\\
      \cdots&\cdots&\cdots&\cdots\\
      \widehat{\sum_{r=1}^{n}(i_k^r-1)}&\widehat{i_k^{n+1}}&\cdots&\widehat{i_{k}^{2n-2}}\\
      \cdots&\cdots&\cdots&\cdots\\
       \sum_{r=1}^{n}(i_{n}^r-1) &i_{n}^{n+1}&\cdots &i_{n}^{2n-2}\\
    \end{array}
  \right)$$

and
   \bigskip $$\tilde{C}_{s}^{(i)}=\left(
    \begin{array}{cccccc}
      i_1^{i}&i_1^{n+1}&\cdots&i_1^{2n-2}\\
      \cdots&\cdots&\cdots&\cdots\\
      \widehat{i_{s}^{i}}&\widehat{i_s^{n+1}}&\cdots&\widehat{i_{s}^{2n-2}}\\
      \cdots&\cdots&\cdots&\cdots\\
       i_{n}^i&i_{n}^{n+1}&\cdots &i_{n}^{2n-2}\\
    \end{array}
  \right)$$


\noindent where the hats mean that the corresponding row is removed.

\section{Main theorems and their proofs}
In this section we will state the main result of this paper. Recall that the inner derivations of the simple linearly compact $n$-Lie algebra $S^{n}$ are isomorphic to $S_{n}$ and denote by $\frak h\, $
 the Cartan subalgebra of the Lie algebra $\frak{sl}_{n}(\mathbb{F})$ chosen  above Theorem \ref{th6}.
Let $F$  be a finite dimensional irreducible highest weight
$\frak{sl}_{n}(\mathbb{F})$-module, with highest weight $\lambda \in \frak h^* $ and
highest weight vector $v_{\lambda.}$
Recall that our goal is to determine for which $\lambda \in \frak h^{\ast},$ the two sided ideal $Q(S^n)$ acts trivially on the irreducible highest weight module $J(F)=M(F)/\hbox{ Sing}_+(M(F)),$ This will ensure us that $J(F)$ is an $n$-Lie module of $S^n.$ Let's denote by $\lambda_i=\lambda(E_{i, i}-E_{i+1,i+1})$  for $ i=1, \cdots ,n-1$
\noindent and introduce the following useful notation  for the proof of the theorem,
\begin{equation}\label{eq19}\delta_{i,j}=
\begin{cases} 1 \;&\hbox{if}\; i\geq j\\
0  &\;\hbox{otherwise},
\end{cases}
\end{equation}
with $i, j \in\{1,\cdots, n\}.$

\begin{theorem} 

 Let $n\geq 3$ and
  $F$ be a finite dimensional  irreducible highest weight  $\frak{sl}_{n}(\mathbb{F})$-module,
    then the irreducible  continuous representation $J(F)$ of $S_{n}$ is an irreducible continuous representation of the simple linearly compact $n$-Lie algebra $S^n$ if and only if $\lambda \in \frak h^{\ast}$ is such that $\lambda=(0,0,\cdots,0).$

\end{theorem}
\begin{proof}
Let $F$  be a highest weight irreducible finite dimensional
$\frak{sl}_{n}(\mathbb{F})$-module, with highest weigh $\lambda \in \frak h^* $ and
highest weigh vector $v_{\lambda.}$ Recall that $\frak h:=\displaystyle{\oplus_{i=1}^{n-1}\mathbb{F}\,(E_{i,i}-E_{i+1,i+1})}$ is the chosen Cartan subalgebra of the Lie algebra $\frak{sl}_{n}(\mathbb{F}).$ Here we are identifying the subalgebra $\frak h$ with the subalgebra of $S_{n}$ generated by the elements $x_i\frac{\partial}{\partial x_i}-x_{i+1}\frac{\partial}{\partial x_{i+1}}, \, i=1, \cdots, n-1.$ 
Consider F as a $(S_{n})_{\geq
0}$-module and take the induced module $M(F)=U(S_{n})\otimes_{U((S_{n})_{\geq 0})} F.$ We will use Lemma \ref{lem4} and the general look of the generators of $Q_A(S^n)$  to find out for which $\lambda$'s, $Q_A(S^n)$ acts trivially in $J(F)$.
 \, Let $w_{\lambda}=1\otimes_{U((S_{n})_{\geq 0})}v_{\lambda}=1\otimes v_{\lambda}.$

According to the description of the generators given in (\ref{eq16}) and taking into account that $(S_{n})_+$ acts by zero on $w_{\lambda},$ it is enough to consider the subset of generators $Q_A(S^n)$ and ask them to either act trivially $w_{\lambda}$ if $F$ is non-exceptional or $Q_A(S^n)\otimes v_{\lambda} \subseteq \hbox{Sing}_{+}(M(F))$ otherwise.  It is enough to consider $x_{f_1, \cdots, f_{2n-2}}$  with monomials $f_i\in \mathbb{F}[x_1,\cdots,x_{n}]$  as in (\ref{eq16}) such that,
\begin{itemize}
\item[(1)] $\hbox{deg}(f_1\cdots
f_{2n-2})=2n-1$ and there exist $i \in \{1, \cdots, n\}$ such that
\begin{itemize}
\item [(a)] $\hbox{deg}(f_if_{n+1}\cdots f_{2n-2})=n-1$ or
\item[(b)]$\hbox{deg}(f_if_{n+1}\cdots f_{2n-2})=n,$
\end{itemize}
\item[(2)] $\hbox{deg}(f_1f_2\cdots
f_{2n-2})=2n,$ and there exist $i \in \{1, \cdots, n\}$ such that $\hbox{deg}(f_if_{n+1}\cdots f_{2n-2})=n$
\item[(3)] $\hbox{deg}(f_1f_2\cdots
f_{2n-2})=2n+1,$ and there exist $i \in \{1, \cdots, n\}$ such that $\hbox{deg}(f_if_{n+1}\cdots f_{2n-2})=n$
\end{itemize}
since the remaining ones are either zero or  act trivially any way.
Here, we are assuming by simplicity that $i=n$ and $f_i \notin \mathbb{F}$ for all $i=1,\cdots, 2n-2$. 
Let's analyze each possible case.

\

 \noindent \underline{Case 1(a)}:

 \

 \noindent Here, $\hbox{deg}(f_1\cdots f_{2n-2})=2n-1,$ $\hbox{deg}(f_1\cdots f_{n-1})=n$ and  $\hbox{deg}(f_n\cdots f_{2n-2})=n-1.$ We have two possible expressions for $f_1\cdots f_{n-1}$ such that $x_{f_1,\cdots, f_{2n-2}}\neq 0$ and two expression for $f_nf_{n+1}\cdots f_{2n-2}.$ Namely, there exist $q,l,j,k, m \in \{1, \cdots, n\}$, such that
  \begin{equation}
 \label{case1.1}f_1\cdots f_{n-1}=x_1 \cdots x_{n}\end{equation}
 or
 \begin{equation}
 \label{case1.2}f_1\cdots f_{n-1}=x_1 \cdots\hat{x_l}\cdots x_m^2\cdots x_{n},\end{equation}
and
 \begin{equation}\label{case1.3}f_n\cdots f_{2n-2}=x_1\cdots \hat{x_k}\cdots x_{n}\end{equation}
 or
 \begin{equation}\label{case1.4}f_n\cdots f_{2n-2}=x_1\cdots \hat{x_q}\cdots \hat{x_j}\cdots x_{k}^2\cdots x_{n}.\end{equation}

\


 \

   Suppose we have (\ref{case1.1}) and (\ref{case1.3}), namely $f_1\cdots f_{n-1}=x_1 \cdots x_{n}$ and $f_n\cdots f_{2n-2}=x_1\cdots \hat{x_k}\cdots x_{n}$ for some $k \in \{1,\cdots, n\}.$ Therefore we can consider the monomials as follows.
   \begin{itemize}
   \item[(i)] Let  $n\geq 3$ and $l,j, k \in \{1,\cdots n\}.$ Note that to define the monomials $f_{n+1},\cdots f_{2n-2}$ we are assuming that $j<k$. Otherwise, we can interchange those indexes in the definition of $f_{n+1},\cdots f_{2n-2}.$
\begin{align*}
&\,f_{s}=x_s, \qquad\quad s=1,\cdots l-1,\nonumber\\
&\,f_{s}=x_{s+1}, \qquad\quad s=l,\cdots n-1, \,\,  s\neq j-\delta_{j,l},\nonumber\\
&\, f_{j-\delta_{j,l}}=x_jx_l,\qquad f_n = x_j,\nonumber\\
&\,f_{n+s}=x_s,\qquad  s=1,\cdots,j-1,\nonumber\\
 &\,f_{n+s}=x_{s+1}, \quad s=j,\cdots k-2,\nonumber\\
 &\,f_{n+s}=x_{s+2}, \quad s=k-1,\cdots n-2.\nonumber\\
\end{align*}
\noindent Thus, using (\ref{eq16}) for these $f_i$'s, it follows that,
\begin{equation}\label{eq23} x_{f_1,\cdots,f_{2n-2}}\cdot (1\otimes v_{\lambda})=(-1)^{j+l+k+\delta_{k,l}} (D_l \otimes E_{l,k}v_{\lambda}-D_j\otimes E_{j,k}v_{\lambda}$$$$-D_k\otimes(E_{l,l}-E_{j,j})v_{\lambda}).\end{equation}

 Now suppose (\ref{case1.1}) and (\ref{case1.4}), namely $f_1\cdots f_{n-1}=x_1 \cdots x_{n}$ and $f_n\cdots f_{2n-2} = x_1\cdots \hat{x_q}\cdots \hat{x_j}\cdots x_{k}^2\cdots x_{n}$ for some $j, q, k \in \{1,\cdots, n\}.$  Therefore, we have the following possibilities.

 \item [(ii)] Let $n\geq 4$ and $q ,l,j, k \in \{1,\cdots n-1\}.$ Note that to define the monomials $f_{n+1},\cdots f_{2n-2}$ we are assuming that $q<j$. Otherwise, we can interchange those indexes in the definition of $f_{n+1},\cdots f_{2n-2}.$
\begin{align*}
&\,f_{s}=x_s, \qquad\quad s=1,\cdots l-1,\,\,  s\neq k-\delta_{k,l},\nonumber\\
&\,f_{s}=x_{s+1}, \qquad\quad s=l,\cdots n-1,\nonumber\\
&\, f_{k-\delta_{k,l}}=x_kx_l,\qquad f_n = x_k,\nonumber\\
&\,f_{n+s}=x_s,\qquad  s=1,\cdots,q-1,\nonumber\\
 &\,f_{n+s}=x_{s+1}, \quad s=q,\cdots j-2,\nonumber\\
 &\,f_{n+s}=x_{s+2}, \quad s=j-1,\cdots n-2.\nonumber\\
\end{align*}
By (\ref{eq16}) we have:
\begin{equation}\label{eq24} x_{f_1,\cdots,f_{2n-2}}\cdot (1\otimes v_{\lambda})=(-1)^{l+q+j+\delta_{q,j}} (D_j \otimes E_{k,q}v_{\lambda}-D_q\otimes E_{k,j}v_{\lambda}).\end{equation}
\end{itemize}
Equation (\ref{case1.2})  combined with equations  (\ref{case1.3}) and (\ref{case1.4}) don't give us new results.

\pagebreak

 \noindent \underline{Case 1(b)}:

 \

Don't provide new information.

\

 \noindent \underline{Case 2}:

 \
 \noindent Here, $\hbox{deg}(f_1\cdots f_{2n-2})=2n,$ $\hbox{deg}(f_1\cdots f_{n-1})=n= \hbox{deg}(f_n\cdots f_{2n-2})$  The two possible expressions for $\hbox{deg}(f_1\cdots f_{n-1})$ such that $x_{f_1,\cdots, f_{2n-2}}\neq 0$ are the same that (\ref{case1.1}) and (\ref{case1.2}). We have three expression for $f_n\cdots f_{2n-2}.$ Namely, there exist $q,j, k, r \in \{1, \cdots, n\}$, such that
 \begin{equation}\label{case2.3}f_n\cdots f_{2n-2}=x_1\cdots x_{n}\end{equation}
 or
 \begin{equation}\label{case2.4}f_n\cdots f_{2n-2}=x_1\cdots \hat{x_j}\cdots x_{k}^2\cdots x_{n}\end{equation}
 or
 \begin{equation}\label{case2.5}f_n\cdots f_{2n-2}=x_1\cdots \hat{x_q}\cdots \hat{x_r}\cdots{x_j}^2\cdots x_{k}^2\cdots x_{n}.\end{equation}

  \noindent Consider (\ref{case1.1}) and (\ref{case2.3}), namely $f_1\cdots f_{n-1}=x_1 \cdots x_{n}=f_n\cdots f_{2n-2}.$   Therefore we can consider the monomials as follows.
   \begin{itemize}
   \item[(i)] Let  $n\geq 4$ and $l, m, j, k \in \{1,\cdots n\},$  Note that to define the monomials $f_{n+1},\cdots f_{2n-2}$ we are assuming that $m<j$. Otherwise, we can interchange those indexes in the definition of $f_{n+1},\cdots f_{2n-2}.$
\begin{align*}
&\,f_{s}=x_s, \qquad\quad s=1,\cdots j-1,\,\,  s\neq l-\delta_{l,j},\nonumber\\
&\,f_{s}=x_{s+1}, \qquad\quad s=j,\cdots n-1,\nonumber\\
&\, f_{l-\delta_{l,j}}=x_lx_j,\qquad f_n = x_j,\nonumber\\
&\,f_{n+s}=x_s,\qquad  s=1,\cdots,m-1,\,\, s\neq k-\delta_{k,m}-\delta_{k, j},\nonumber\\
&\, f_{k-\delta_{k,m}-\delta_{k, j}}=x_kx_m, \nonumber\\
 &\,f_{n+s}=x_{s+1}, \quad s=m,\cdots j-2,\nonumber\\
 &\,f_{n+s}=x_{s+2}, \quad s=j-1,\cdots n-2.\nonumber\\
\end{align*}
\noindent Thus, using (\ref{eq16}) for these $f_i$'s, it follows that,
\begin{equation}\label{eq28} x_{f_1,\cdots,f_{2n-2}}\cdot (1\otimes v_{\lambda})=(-1)^{m+\delta_{m,j}} (1\otimes E_{j,m}E_{m,j}v_{\lambda}-1\otimes E_{jk}E_{kj}v_{\lambda}$$$$+1\otimes(E_{m,m}-E_{k,k})(1-(E_{j,j}-E_{l,l}))v_{\lambda}).\end{equation}

Suppose in the definitions of $f_1, \cdots f_{2n-2}$ in (i), we take $m := l,$ then we have:
 \item [(ii)]Let  $n\geq 3$ and $j, k,l\in \{1,\cdots n\},$
\begin{equation}\label{eq29} x_{f_1,\cdots,f_{2n-2}}\cdot (1\otimes v_{\lambda})=(-1)^{l+\delta_{j,l}} (1\otimes E_{jk}E_{kj}v_{\lambda}$$$$-1\otimes (E_{l,l}-E_{k,k})(1-(E_{j,j}-E_{l,l}))v_{\lambda}).\end{equation}
\
Now, suppose in the definitions of $f_1, \cdots f_{2n-2}$ in (i), we take $l: = k,$ then we have:
 \item[(iii)] Let  $n\geq 3$ and $m,j, k\in \{1,\cdots n\},$ 
\begin{equation}\label{eq30} x_{f_1,\cdots,f_{2n-2}}\cdot (1\otimes v_{\lambda})=(-1)^{m+\delta_{m,j}} (1\otimes (E_{j,j}-E_{k,k})(1-(E_{m,m}-E_{k,k}))v_{\lambda}).\end{equation}

Suppose in the definitions of $f_1, \cdots f_{2n-2}$ in(iii), we take $l: = j, \, j:=m$ and $m:=q$ then we have:
 \item[(iv)] Let  $n\geq 4$ and $q,m,j, k\in \{1,\cdots n\},$ 
\begin{equation}\label{eq31} x_{f_1,\cdots,f_{2n-2}}\cdot (1\otimes v_{\lambda})=(-1)^{q+m+j+\delta_{j,q}} (1\otimes (E_{m,m}-E_{j,j})(E_{q,q}-E_{k,k})v_{\lambda}).\end{equation}

 Now consider (\ref{case1.1}) and (\ref{case2.4}), namely $f_1\cdots f_{n-1}=x_1\cdots x_{n}$ and $f_n\cdots f_{2n-2}=x_1\cdots \hat{x_j}\cdots x_{k}^2\cdots x_{n}$.
  \item[(v)] Let  $n\geq 5$ and $q,m,j, k,l\in \{1,\cdots n\}.$  Note that to define the monomials $f_{n+1},\cdots f_{2n-2}$ we are assuming that $q<j$. Otherwise, we can interchange those indexes in the definition of $f_{n+1},\cdots f_{2n-2}.$
\begin{align*}
&\,f_{s}=x_s, \qquad\quad s=1,\cdots k-1,\nonumber\\
&\,f_{s}=x_{s+1}, \qquad\quad s=k,\cdots n-1,\,\,  s\neq l-\delta_{l,j}, \nonumber\\
&\, f_{l-\delta{l,j}}=x_lx_k,\qquad f_n = x_k,\nonumber\\
&\,f_{n+s}=x_s,\qquad  s=1,\cdots,q-1,\,\, s\neq m-\delta_{m,q}-\delta_{m, j},\nonumber\\
&\, f_{m-\delta_{m,q}-\delta_{m, j}}=x_mx_q, \nonumber\\
 &\,f_{n+s}=x_{s+1}, \quad s=q,\cdots j-2,\nonumber\\
 &\,f_{n+s}=x_{s+2}, \quad s=j-1,\cdots n-2.\nonumber\\
\end{align*}
\noindent Thus, using (\ref{eq16}) for these $f_i$'s, it follows that,
\begin{equation}\label{eq32} x_{f_1,\cdots,f_{2n-2}}\cdot (1\otimes v_{\lambda})=(-1)^{(j+q+k+\delta_{q,k})} (1\otimes E_{k,q}E_{q,j}v_{\lambda}-1\otimes E_{k,m}E_{m,j}v_{\lambda}$$$$-1\otimes E_{kj}(E_{q,q}-E_{m,m})v_{\lambda}).\end{equation}
Suppose in the definitions of $f_1, \cdots f_{2n-2}$ in (v) we take $m := l,$ then we have:
 \item[(vi)] Let  $n\geq 4$ and $m,j, k,l\in \{1,\cdots n\},$  
\begin{equation}\label{eq33} x_{f_1,\cdots,f_{2n-2}}\cdot (1\otimes v_{\lambda})=(-1)^{(j+m+k+\delta_{m,k})} (1\otimes E_{k,q}E_{q,j}v_{\lambda}$$$$-1\otimes E_{kj}(E_{q,q}-E_{l,l})v_{\lambda}).\end{equation}
Now, suppose in the definitions of $f_1, \cdots f_{2n-2}$ in (v), we take $ l:=j,$ then we have:
\item[(vii)] Let  $n\geq 4$ and $m,j, k,l\in \{1,\cdots n\},$ 
\begin{equation}\label{eq34} x_{f_1,\cdots,f_{2n-2}}\cdot (1\otimes v_{\lambda})=(-1)^{(l+m+k+\delta_{m,l})} (1\otimes E_{k,q}E_{q,j}v_{\lambda}-1\otimes E_{km}E_{m,j}v_{\lambda}).\end{equation}


\item[(viii)] Let  $n\geq 5$ and $m, q, j, k,l\in \{1,\cdots n\}.$  Note that to define the monomials $f_{n+1},\cdots f_{2n-2}$ we are assuming that $q<j$. Otherwise, we can interchange those indexes in the definition of $f_{n+1},\cdots f_{2n-2}.$
\begin{align*}
&\,f_{s}=x_s, \qquad\quad s=1,\cdots m-1,\,\,  s\neq l-\delta_{l,m},\nonumber\\
&\,f_{s}=x_{s+1}, \qquad\quad s=m,\cdots n-1,\nonumber\\
&\, f_{l-\delta_{l,m}}=x_lx_m,\qquad f_n = x_kx_q,\nonumber\\
&\,f_{n+s}=x_s,\qquad  s=1,\cdots,q-1,\nonumber\\
 &\,f_{n+s}=x_{s+1}, \quad s=q,\cdots j-2,\nonumber\\
 &\,f_{n+s}=x_{s+2}, \quad s=j-1,\cdots n-2.\nonumber\\
\end{align*}
\noindent Thus, using (\ref{eq16}) for these $f_i$'s, it follows that,
\begin{equation}\label{eq35} x_{f_1,\cdots,f_{2n-2}}\cdot (1\otimes v_{\lambda})=(-1)^{(m+j+q+\delta_{q,j})} (1\otimes E_{k,j}(E_{m,m}-E_{l,l})v_{\lambda}).\end{equation}
Suppose in the definitions of $f_1, \cdots f_{2n-2}$ in (viii) we take $ l:=k,$ then we have:
\item[(ix)] Let  $n\geq 4$ and $m, j, k,l\in \{1,\cdots n\},$  
\begin{equation}\label{eq36} x_{f_1,\cdots,f_{2n-2}}\cdot (1\otimes v_{\lambda})=(-1)^{(m+j+l+\delta_{l,j})} (1\otimes E_{k,j}(E_{m,m}-E_{k,k}-1)v_{\lambda}).\end{equation}
\item[(x)] Let  $n\geq 3$ and $l, j, k\in \{1,\cdots n\}.$  Note that to define the monomials $f_{n+1},\cdots f_{2n-2}$ we are assuming that $l<j$. Otherwise, we can interchange those indexes in the definition of $f_{n+1},\cdots f_{2n-2}.$
\begin{align*}
&\,f_{s}=x_s, \qquad\quad s=1,\cdots l-1,\,\,  s\neq j-\delta_{j,l},\nonumber\\
&\,f_{s}=x_{s+1}, \qquad\quad s=l,\cdots n-1,\nonumber\\
&\, f_{j-\delta_{j,l}}=x_lx_j,\qquad f_n = x_l,\nonumber\\
&\,f_{n+s}=x_s,\qquad  s=1,\cdots,l-1,\, s\neq k -\delta_{k,l}-\delta_{k,j},\nonumber\\
&\, f_{n+k-\delta_{k, l}-\delta_{k,j}}=x_k^2,\nonumber\\
 &\,f_{n+s}=x_{s+1}, \quad s=l,\cdots j-2,\nonumber\\
 &\,f_{n+s}=x_{s+2}, \quad s=j-1,\cdots n-2.\nonumber\\
\end{align*}
\noindent Thus, using (\ref{eq16}) for these $f_i$'s, it follows that,
\begin{equation}\label{eq37} x_{f_1,\cdots,f_{2n-2}}\cdot (1\otimes v_{\lambda})=(-1)^{(j+\delta_{j,l})} (1\otimes E_{k,j}(E_{l,l}-E_{j,j})v_{\lambda}).\end{equation}
\end{itemize}

Suppose (\ref{case1.1}) and (\ref{case2.5}), namely $f_1\cdots f_{n-1}= x_1\cdots x_n$ and
 $f_n\cdots f_{2n-2}=x_1\cdots \hat{x_q}\cdots \hat{x_r}\cdots{x_j}^2\cdots x_{k}^2\cdots x_{n},$ for some $q,r,j,k \in \{1,\cdots, n\}.$
  \begin{itemize}
  \item[(xi)] Let  $n\geq 5$ and $r, q, j, k,l\in \{1,\cdots n\}.$  Note that to define the monomials $f_{n+1},\cdots f_{2n-2}$ we are assuming that $r<q$. Otherwise, we can interchange those indexes in the definition of $f_{n+1},\cdots f_{2n-2}.$
\begin{align*}
&\,f_{s}=x_s, \qquad\quad s=1,\cdots l-1,\nonumber\\
&\,f_{s}=x_{s+1}, \qquad\quad s=l,\cdots n-1,\,\,  s\neq j-\delta_{j,l},\nonumber\\
&\, f_{j-\delta_{j,l}}=x_jx_l,\qquad f_n = x_j,\nonumber\\
&\,f_{n+s}=x_s,\qquad  s=1,\cdots,r-1,\, s\neq k-\delta_{k,r}-\delta_{k,q},\nonumber\\
&\,f_{n+k-\delta_{k,r}-\delta_{k,q}}=x^2_k,\nonumber\\
 &\,f_{n+s}=x_{s+1}, \quad s=r,\cdots q-2,\nonumber\\
 &\,f_{n+s}=x_{s+2}, \quad s=q-1,\cdots n-2.\nonumber\\
\end{align*}
\noindent Thus, using (\ref{eq16}) for these $f_i$'s, it follows that,
\begin{equation}\label{eq38} x_{f_1,\cdots,f_{2n-2}}\cdot (1\otimes v_{\lambda})=(-1)^{(l+q+\delta_{q,r})} 2 (1\otimes E_{j,r}E_{k,q}v_{\lambda}-1\otimes E_{j,q}E_{k,r}v_{\lambda}).\end{equation}
\end{itemize}


\noindent Equations (\ref{case1.2}) and (\ref{case2.3}) don't give us new information.
  Now, consider (\ref{case1.2}) and (\ref{case2.4}), namely $f_1\cdots f_{n-1}=x_1 \cdots\hat{x_l}\cdots x_m^2\cdots x_{n}$ and $f_n\cdots f_{2n-2}=x_1\cdots \hat{x_j}\cdots x_{k}^2\cdots x_{n},$ for some $l, m, j, k \in \{1,\cdots, n\}.$

Suppose $m := k$ and $l := q$ in the equation (\ref{case1.2}), then we have:
  \begin{itemize}
\item[(xii)] Let  $n\geq 4$ and $r, q, j, k\in \{1,\cdots n\}.$  Note that to define the monomials $f_{n+1},\cdots f_{2n-2}$ we are assuming that $q<j$. Otherwise, we can interchange those indexes in the definition of $f_{n+1},\cdots f_{2n-2}.$
\begin{align*}
&\,f_{s}=x_s, \qquad\quad s=1,\cdots k-1,\nonumber\\
&\,f_{s}=x_{s+1}, \qquad\quad s=k,\cdots n-1,\,\,  s\neq q-\delta_{q,k},\nonumber\\
&\, f_{q-\delta_{q,k}}=x^2_q,\qquad f_n = x_k,\nonumber\\
&\,f_{n+s}=x_s,\qquad  s=1,\cdots,q-1,\, s\neq r-\delta_{r,q}-\delta_{r,j},\nonumber\\
&\,f_{n+r-\delta_{r,q}-\delta_{r,j}}=x_rx_q,\qquad  \nonumber\\
 &\,f_{n+s}=x_{s+1}, \quad s=q,\cdots j-2,\nonumber\\
 &\,f_{n+s}=x_{s+2}, \quad s=j-1,\cdots n-2.\nonumber\\
\end{align*}
\noindent Thus, using (\ref{eq16}) for these $f_i$'s, it follows that,
\begin{equation}\label{eq39} x_{f_1,\cdots,f_{2n-2}}\cdot (1\otimes v_{\lambda})=(-1)^{(k+j+q+\delta_{j,q})} (1\otimes E_{q,r}E_{r,j}v_{\lambda}$$$$+ 1\otimes E_{q, j}(E_{q,q}-E_{r,r}+1)v_{\lambda}).\end{equation}

Suppose  in the definitions of $f_1,\cdots, f_{2n-2}$ in the equation (\ref{case1.1}), we take $ m:= k$ and $l:= j$, then we have:
\item[(xiii)] Let  $n\geq 4$ and $r, q, j, k\in \{1,\cdots n\}.$  Note that to define the monomials $f_{n+1},\cdots f_{2n-2}$ we are assuming that $q<j$. Otherwise, we can interchange those indexes in the definition of $f_{n+1},\cdots f_{2n-2}.$
\begin{align*}
&\,f_{s}=x_s, \qquad\quad s=1,\cdots k-1,\nonumber\\
&\,f_{s}=x_{s+1}, \qquad\quad s=k,\cdots n-1, \,\,  s\neq j-\delta_{j,k},\nonumber\\
&\, f_{j-\delta_{j,k}}=x^2_j,\qquad f_n = x_k,\nonumber\\
&\,f_{n+s}=x_s,\qquad  s=1,\cdots,q-1,\, s\neq r-\delta_{r,q}-\delta_{r,j},\nonumber\\
&\,f_{n+r-\delta_{r,q}-\delta_{r,j}}=x_rx_q,\qquad  \nonumber\\
 &\,f_{n+s}=x_{s+1}, \quad s=q,\cdots j-2,\nonumber\\
 &\,f_{n+s}=x_{s+2}, \quad s=j-1,\cdots n-2.\nonumber\\
\end{align*}
\noindent Thus, using (\ref{eq16}) for these $f_i$'s, it follows that,
\begin{equation}\label{eq40} x_{f_1,\cdots,f_{2n-2}}\cdot (1\otimes v_{\lambda})=(-1)^{(k+j+q+\delta_{j,q})} (1\otimes E_{q,j}E_{j,q}v_{\lambda}-1\otimes E_{r,j}E_{j,r}v_{\lambda}$$$$+ 1\otimes (E_{q,q}-E_{r,r})v_{\lambda}).\end{equation}

  \item[(xiv)] Let  $n\geq 5$ and $m, q, j, k,l\in \{1,\cdots n\}.$  Note that to define the monomials $f_{n+1},\cdots f_{2n-2}$ we are assuming that $q<j$. Otherwise, we can interchange those indexes in the definition of $f_{n+1},\cdots f_{2n-2}.$
\begin{align*}
&\,f_{s}=x_s, \qquad\quad s=1,\cdots m-1,\nonumber\\
&\,f_{s}=x_{s+1}, \qquad\quad s=m,\cdots n-1, \,\,  s\neq l-\delta_{l,m},\nonumber\\
&\, f_{l-\delta_{l,m}}=x^2_l,\qquad f_n = x_kx_q,\nonumber\\
&\,f_{n+s}=x_s,\qquad  s=1,\cdots,q-1,\nonumber\\
 &\,f_{n+s}=x_{s+1}, \quad s=q,\cdots j-2,\nonumber\\
 &\,f_{n+s}=x_{s+2}, \quad s=j-1,\cdots n-2.\nonumber\\
\end{align*}
\noindent Thus, using (\ref{eq16}) for these $f_i$'s, it follows that,
\begin{equation}\label{eq41} x_{f_1,\cdots,f_{2n-2}}\cdot (1\otimes v_{\lambda})=(-1)^{(l+q+\delta_{r,q})} (1\otimes E_{l,m}E_{k,j}v_{\lambda}).\end{equation}
Suppose  in the definitions of $f_1,\cdots, f_{2n-2}$ in the example (xiv), we take $l:= k,$ then we have.
  \item[(xv)] Let  $n\geq 4$ and $m, q, j, k\in \{1,\cdots n\},$
\begin{equation}\label{eq42} x_{f_1,\cdots,f_{2n-2}}\cdot (1\otimes v_{\lambda})=(-1)^{(q+\delta_{q,k})} (1\otimes E_{k,m}E_{k,j}v_{\lambda}).\end{equation}
Now, suppose  in the definitions of $f_1,\cdots, f_{2n-2}$ in (xiv) we take $l:= k$ and $m:=j$, then we have:

  \item[(xvi)] Let  $n\geq 3$ and $, q, j, k\in \{1,\cdots n\},$  
\begin{equation}\label{eq43} x_{f_1,\cdots,f_{2n-2}}\cdot (1\otimes v_{\lambda})=(-1)^{(q+\delta_{q,k})} (1\otimes E_{k,j}E_{k,j}v_{\lambda}).\end{equation}
Suppose  in the definitions of $f_1,\cdots, f_{2n-2}$ in (xiv) we take $m:=k$ then we have:

  \item[(xvii)] Let  $n\geq 4$ and $, q, j, k,l\in \{1,\cdots n\},$ 
\begin{equation}\label{eq44} x_{f_1,\cdots,f_{2n-2}}\cdot (1\otimes v_{\lambda})=(-1)^{(j+k+q+\delta_{j,q})} (1\otimes E_{l,j}v_{\lambda}-1\otimes E_{k,j}E_{l,k}v_{\lambda}).\end{equation}
\end{itemize}
 Equations (\ref{case1.2}) and (\ref{case2.5}) don't give us new information.

 \pagebreak

\noindent\underline{Case (3)}

\

Don't give us new equations.

\

Observe that  the right hand side of all the equations (\ref{eq28}) to (\ref{eq44}) belongs to $1\otimes_{U((S_{n})_{\geq 0})},$ therefore they are trivial singular vectors. Due to Lemma \ref{lem4} and the fact that $\hbox{ Sing}_+(M(F))$  does not contain trivial singular vectors , we need to ensure that all the equations (\ref{eq28}) to (\ref{eq44}) are equal to zero.
Since different equations hold for $n=3$ and $n\geq 4$ ,  we will study theses cases separately.

If $n=3,$ equations (\ref{eq29}), (\ref{eq30}), (\ref{eq37}) and ($\ref{eq43}$) hold and they have to be zero. Note that equation (\ref{eq30}) is equivalent to
\begin{equation}\label{eq55}-(\lambda_1+\lambda_2)(1-\lambda_2)(1\otimes v_{\lambda})=0,\end{equation}
\begin{equation}\label{eq56}\lambda_1(1+\lambda_2)(1\otimes v_{\lambda})=0,
\end{equation}
\begin{equation}\label{eq57}\lambda_1(1+\lambda_1+\lambda_2)(1\otimes v_{\lambda})=0.
\end{equation}

 Thus equations (\ref{eq55}) to (\ref{eq57}) implies $\lambda_1=\lambda_2=0$ or $\lambda_1=0, \, \lambda_2=1.$ 
%

\

Now, if $n\geq 4,$ equations (\ref{eq30}) and ($\ref{eq32}$) equate to zero implies that $\lambda_1=\lambda_2=\cdots=\lambda_{n-1}=0$ or $\lambda_1=\lambda_2=\cdots=\lambda_{n-2}=0$ and $\lambda_{n-1}=1.$

 Then we will apply the Freudenthal's formula to calculate the dimensions of the weight spaces and check wether the remaining equations are satisfied.

 We will need the following notation to apply Freudental's formula to $\frak{sl}_n(\mathbb{F})$ (cf. Section 22.3 in \cite{H}):

Let $\frak h$ be our chosen  Cartan subalgebra of $\frak{sl}_{n}(\mathbb{F})$ and  $\epsilon_j$ be defined by $\epsilon_{j}\left(\sum\limits_{i=1}^{n}a_i E_{i,i}\right)=a_j.$
We will consider  the roots $$\phi=\{\epsilon_i-\epsilon_{j}\mid 1\leq i\neq j\leq n\}$$
where the  root space  associated to $(\epsilon_i-\epsilon_{j})$ is generated by $E_{i,j}$ and simple roots are  $$\Delta=\{\epsilon_1-\epsilon_{2}, \epsilon_2-\epsilon_{3},\cdots,\epsilon_{n-1}-\epsilon_{n} \}.$$
Let $\Lambda^+$ be the set of all dominant weights and $\delta = \frac{1}{2}\sum_{\alpha \succ 0}\alpha$.
If $\alpha_{i}:=\epsilon_i-\epsilon_{i+1},$ the fundamental dominant weights relatives to $\Delta$ of $\frak{sl}_{n}(\mathbb{F})$ are given by
\begin{eqnarray}\label{eq4}\pi
_i=&\frac{1}{n}[(n-i)\alpha_1+2(n-i)\alpha_2+\cdots (i-1)(n-i)\alpha_{i-1}\nonumber\\&+i(n-i)\alpha_i
+i(n-1-i)\alpha_{i+1}+\cdots+i\alpha_{n-1}].\end{eqnarray}
Therefore $\Lambda$ is a lattice with basis $\pi_i,\,i=1, \cdots,n-1 .$

Let $n\geq 3$. 
Require that $(\alpha_i, \alpha_i)=1,$ $(\alpha_i,\alpha_j)=-1/2$
if $\mid i-j\mid=1$ and $(\alpha_i,\alpha_j)=0$ if $\mid i-j\mid\geq 2.$

First we will consider $\lambda=(\lambda_1, \cdots, \lambda_{n-1})=(0,\cdots, 0).$ Since $(\lambda+\delta, \lambda+\delta)-(\mu+\delta,\mu+\delta)=0$ for $\mu=-\alpha_{k-1}$  with $k\in \{1, \cdots, n\}$ it follows from Proposition 21.3 and Lemma C of (13.4) in \cite{H}, that $\mu$ is not a weight, therefore multiplicities $\mu=-\alpha_{k-1}$ is equal to zero. Also, it follows from Freudenthal´s formula, that the multiplicities for  $\mu=-\sum_{k=j}^{i-1}\alpha_k$ are also equal to zero for all $i,j \in \{1, \cdots, n\}, \, i>j.$ Thus $E_{i, j}v_{\lambda}=0,$ for all $i>j $. In particular  all the equations from (\ref{eq23}) to (\ref{eq44}) are equal to zero. Observe that if $\lambda=(\lambda_1, \cdots, \lambda_{n-1})=(0, \cdots, 0),$ for $n\geq 3$ then the $\frak{sl}_n(\mathbb{F})$-module $F$ coincides with the exceptional module $F_{0}.$ Due to Theorem \ref{th6} we have to take the quotient of $M(F_{0})$ by the submodule generated by all its non-trivial singular vectors to make the module irreducible.

Finally, if $\lambda=(0,0,\cdots,1),$  the Freudenthal's formula gives that the multiplicities for $\mu=-2(\alpha_{n-2}+\alpha_{n-1})$ are equal to one. This implies that $E_{n, n-2}E_{n, n-2}v_{\lambda}\neq 0,$ therefore equation $(\ref{eq43})$ is non zero and the induced representation $M(F)$ is not a representation of the $n$-Lie algebra $S^n,$ for $n\geq 3$. Conversely, it is straightforward to check that  if $\lambda=(\lambda_1, \cdots, \lambda_{n-1})=(0,\cdots, 0)$, the corresponding irreducible quotient of the induced module is an $S^n$ module, finishing our proof.
\end{proof}


\begin{thebibliography}{10}

\bibitem [BL]{BL}  Dana Balibanu, Johan van de Leur, {\em Irreducible Highest Weight
Representations of The Simple n-Lie Algebra},Transformation Groups
{\bf 17} (3), (2012) 593-613.
\bibitem [BL1]{BL2}  Dana Balibanu, Johan van de Leur, Erratum {\em Irreducible Highest Weight
Representations of The Simple n-Lie Algebra},Transformation Groups
{\bf 21} (1), (2016) 297.
\bibitem [CK]{CK}  Nicoletta Cantarini, Victor G. Kac, {\em Clasisification of simple linearly compact n-Lie superalgebras},  Comm. Math. Phys. {\bf 298} (2010), 833-853.

\bibitem [D]{D}A.S.Dzhumadildaev {\em Identities and derivations for Jacobian algebras}, Contemp.Math. {\bf 315} (2002),
245-278.
\bibitem [D1]{D1} A.S.Dzhumadildaev {\em Representatios of vector product n-Lie algebras}, Comm.Algebra {\bf 32}(9) (2004),3315-3326.

\bibitem [F]{F}V.T.Filippov {\em n-Lie algebras}, Sib.Mat. Zh. {\bf 26}(6) (1985), 126-140, translation in Sib. Math.J. {\bf 26} (6)(1985), 879-891.

\bibitem [F1] {F1}V.T.Filippov {\em  On n-Lie algebras of Jacobians}, Sib.Mat. Zh. {\bf 39}(3) (1998), 660-669, translation in Sib. Math.J. {\bf 39} (3)(1998), 573-581.
\bibitem [H]{H}   J. E Humphreys,  {\em Introduction to Lie algebras and Representation Theory}. (1972) by Springer-Verlag New York Inc. Library of Congress Catalog Card Number 72-85951
\bibitem [K] {K} S.M.Kasymov, {\em On the theory of n-Lie algebras,} Algebra i Logika {\bf 26}(3) (1987), 155-166.

\bibitem [K1]  {K1} S.M.Kasymov, {\em On nil-elements and nil-subsets,} Sib.Mat.J.{\bf 32}(6) (1991), 77-80.
\bibitem[KR]{KR1} V. G. Kac and A. Rudakov,  {\em
Representations of the exceptional Lie superalgebra $E(3,6)$. I.
Degeneracy conditions}, Transform. Groups {\bf 7} (2002), no. 1,
67--86.
\bibitem [L]{L} W.X.Ling, {\em On the structure of n-Lie algebras}, PhD thesis, Siegen, 1993.

\bibitem [N] {N} Y.Nambu, {\em Generalized Hamiltonian mechanics,}
Phys.Rev. {\bf D7} (1973), 2405-2412.

\bibitem [R]{R}   A. Rudakov,  {\em Irreducible representations of
infinite-dimensional Lie algebras of type S and H}.Izv. Akad. Nauk
SSSR Ser. Mat. Vol.9, No.3 (1975), 465--480.


%

\end{thebibliography}
\end{document}